\documentclass[10.5pt,reqno,a4paper,oneside]{article}

\usepackage{breakurl}             
\usepackage{underscore}           

\usepackage[english]{babel}

\usepackage[a4paper]{geometry}

\usepackage{amssymb}
\usepackage{amsthm}
\usepackage{amsmath}
\usepackage{textcmds}
\usepackage{listings}
\usepackage{comment}
\usepackage{tikz-cd} 
\usepackage[ruled,vlined]{algorithm2e}
\usepackage{xcolor}
\usepackage{geometry}
\usepackage[mathlines]{lineno}

\definecolor{customblue}{HTML}{166dde}
\usepackage[colorlinks=true, allcolors = customblue]{hyperref}
\usepackage[skip=10pt]{parskip}

\usepackage{subfig}
\usepackage{caption}

\usepackage{thmtools}
\usepackage{thm-restate}
\usepackage{tikz}
\usepackage{booktabs}

\makeatletter
\let\@internalcite\cite
\def\cite{\def\citeauthoryear##1##2{##1, ##2}\@internalcite}
\def\shortcite{\def\citeauthoryear##1{##2}\@internalcite}
\def\@biblabel#1{\def\citeauthoryear##1##2{##1, ##2}[#1]\hfill}
\makeatother

\newcommand{\R}{\mathbb{R}}

\newcommand{\M}{\mathcal{M}}
\newcommand{\A}{\mathcal{A}}

\newcommand{\ord}{\mathcal{O}}

\newcommand{\thupper}{^{\text{th}}}

\newtheorem{theorem}{Theorem}[section]
\newtheorem{corollary}[theorem]{Corollary}  

\theoremstyle{definition}

\theoremstyle{definition}

\theoremstyle{definition}
 
\theoremstyle{remark}
\newtheorem{remark}[theorem]{Remark}         

\makeatletter
\renewcommand{\thm@space@setup}{%
  \thm@preskip=\parskip
  \thm@postskip=\parskip
}
\makeatother

\title{
Manifold Diffusion Geometry: \\ Curvature, Tangent Spaces, and Dimension}
\author{Iolo Jones \\ Durham University}
\date{December 2024}

\begin{document}
\maketitle

\begin{abstract}
We introduce novel estimators for computing the curvature, tangent spaces, and dimension of data from manifolds, using tools from diffusion geometry.
Although classical Riemannian geometry is a rich source of inspiration for geometric data analysis and machine learning, it has historically been hard to implement these methods in a way that performs well statistically.
Diffusion geometry lets us develop Riemannian geometry methods that are accurate and, crucially, also extremely robust to noise and low-density data.
The methods we introduce here are comparable to the existing state-of-the-art on ideal dense, noise-free data, but significantly outperform them in the presence of noise or sparsity.
In particular, our dimension estimate improves on the existing methods on a challenging benchmark test when even a small amount of noise is added.
Our tangent space and scalar curvature estimates do not require parameter selection and substantially improve on existing techniques.
\end{abstract}

\tableofcontents

\section{Introduction}

For many problems in geometric and topological data analysis, we can safely assume that the data lie on or near a manifold.
This assumption is called the \q{manifold hypothesis} and means that the classical theory of Riemannian geometry can, in principle, be applied to the problem.
However, on a manifold, it is possible to take limits of quantities, and so define derivatives, vector fields, integrals, and all the other calculus tools that give Riemannian geometry its power.
Applying this paradigm directly to data is hard: a finite data set is discrete, so limits cannot exist in the same way, and real-world data is often noisy.

This paper introduces a range of novel techniques for manifold data, using the theory of \textit{diffusion geometry} \cite{jones2024diffusion, jones2026computing}.
Diffusion geometry recasts the classical theory of Riemannian geometry in terms of the \q{heat flow} on the manifold, which is an example of a stochastic process called a \textit{Markov diffusion}.
Markov diffusion processes provide extremely robust statistics and powerful computational tools: they have recently underpinned the development of generative image and video artificial intelligence models \cite{sohl2015deep, ho2020denoising, song2020score}, and there are many well-developed techniques for estimating the heat flow on a manifold \cite{COIFMAN20065, berry2016variable}.
By estimating the heat flow from data we obtain an estimate for the Laplacian operator $\Delta$, and with it can construct most of the important objects in Riemannian geometry.
In particular, we will construct new estimators for the tangent spaces to the manifold, the dimension of the manifold at each point and also globally, and the Riemann, Ricci, and scalar curvatures (see Figure \ref{fig: intro}).
These inherit the excellent statistical properties of the underlying Markov diffusion, and so are extremely robust to noise, outliers, and low sample density.

\begin{figure}[!h]
    \centering
    \includegraphics[width=1\linewidth]{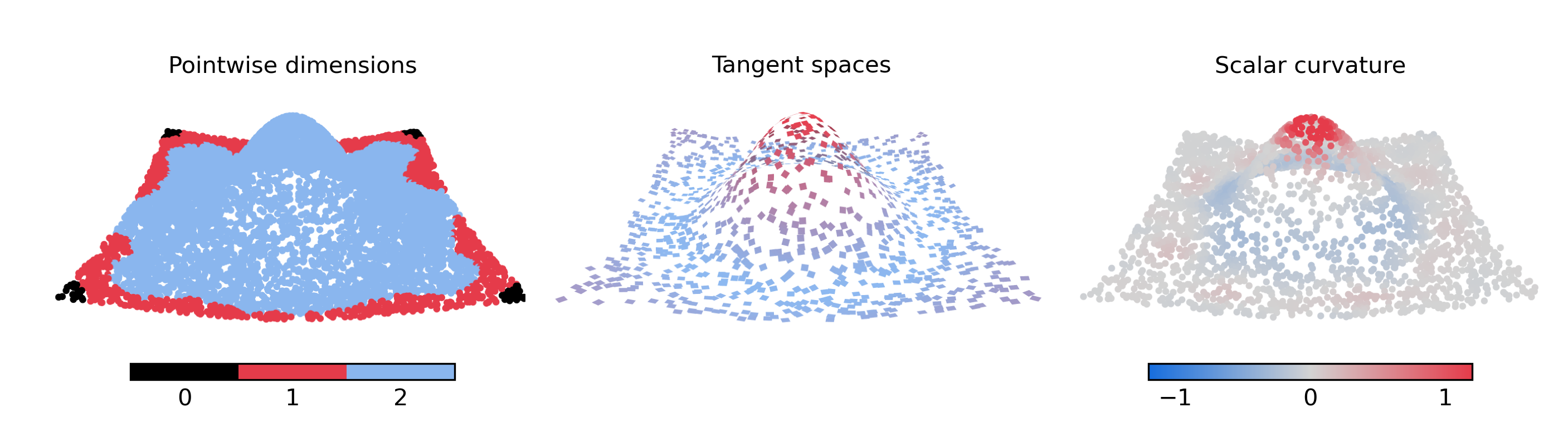}
    \captionsetup{width=0.8\linewidth}
    \caption{The pointwise dimensions, tangent spaces, and scalar curvature of a manifold.}
    \label{fig: intro}
\end{figure}

The problems of dimension and tangent space estimation are very well studied: we compare our estimates to the state-of-the-art and find that, for \textit{high quality} data where the sample is large and the noise level is low, our methods achieve comparable results, but, crucially, for \textit{low quality} sparse and noisy data, our methods significantly outperform all existing approaches.
There is comparatively little literature on curvature estimation: our method is (to our knowledge) only the third to compute scalar curvature on a general manifold \cite{sritharan2021computing,hickok2023intrinsic}, the second to compute the Ricci curvature \cite{ache2022approximating}, and the first to compute the Riemann curvature.
Like \cite{sritharan2021computing}, we will take an \textit{extrinsic} approach and compute curvature using the Hessian of the ambient coordinates (second fundamental form).
This has a significant regularising effect that improves performance, but means that such approaches apply less broadly than the purely intrinsic approach of \cite{hickok2023intrinsic} that does not need coordinates, and may extend to non-manifold data.
In quantitative comparisons, we again find that, although similar on extremely dense and clean data, our new diffusion geometry estimators significantly outperform the others when faced with lower density and noise.

\subsection{Avoiding the \q{hard neighbourhood} paradigm
}

Most existing techniques for computational geometry and topology use combinatorial objects like graphs or simplicial complexes as a model for geometry.
In other words, they resolve the paradox of defining limits on data by defining a \q{neighbourhood} to a point as the set of all points within some given distance.
This can be used to measure global geometry or topology, or local data like tangents and the Hessian through linear regression on the neighbourhood.
A strong theme that emerges consistently throughout this work is that methods like these contain two inherent issues:
\begin{enumerate}
    \item There is a hard cut-off for the edge of a neighbourhood, so points move into and out of each other's neighbourhoods discontinuously.
    It is, therefore, hard to make these methods robust to noise.
    \item There is a trade-off between precision (by taking a smaller neighbourhood) and robustness to noise (by taking a larger neighbourhood).
    This trade-off is usually controlled by a parameter which must be set by the user (and is quantified in \cite{lim2024tangent}).
\end{enumerate}
We will call this approach the \textit{hard neighbourhood paradigm}.
Some \textit{hard neighbourhood} methods, like most of the existing dimension estimation techniques we test below, as well as persistent homology \cite{robins1999towards, edelsbrunner2002topological, zomorodian2004computing}, mitigate the second issue by measuring behaviour over a range of neighbourhood sizes.
This improves performance and can guarantee continuity \cite{cohen2005stability}, but does not avoid the inherently poor robustness of hard neighbourhoods, as we saw for persistent homology in \cite{jones2024diffusion}, and as we will see repeatedly below.

Conversely, diffusion geometry encodes the \q{locality} information in the Gaussian heat kernel (the transition probability of the diffusion process).
Rather than setting a hard cut-off for the edge of a neighbourhood, the kernel smoothly incorporates all the data into and out of consideration, which makes it far more robust to noise.
The kernel does have a bandwidth parameter for the size of the \q{neighbourhood}, but, because that neighbourhood is soft, the results do not depend as sensitively on the choice of bandwidth.
We will use \textit{variable bandwidth} kernels with straightforward automatic parameter selection, and so will suffer from neither of the two problems outlined above.
Our methods are then effectively parameter-free, and highly robust to noise while still being very accurate.

In the infinite-data limit the \textit{hard neighbourhood} and \textit{soft neighbourhood} paradigms converge, but, as statistics on finite data, they are very different.

\subsection{Apology for low intrinsic dimensions}
\label{sub: apology}

Even though the methods presented here are valid in any dimension, we will only test and exemplify them on manifolds of at most 3 dimensions.
This is partly because many of the important applications of computational geometry are to spatial data, which is at most 3-dimensional.
Mainly, however, it is because this paper is about \textit{robust} geometric inference in the presence of noise and outliers, and the extent to which we can tell the signal from the noise depends on the density of the data.

For example, a circle could clearly be identified from a sample of 50 points even with significant noise because 50 points are a dense enough sample for a circle.
One could probably identify a sphere from a sample of 50 points, but this analysis would not survive the addition of much noise at all, because, for a 2-dimensional object, a sample of 50 points is suddenly far too sparse.
It would take a similar density of data, perhaps $50^2 = 2,500$ points, to achieve the same level of certainty as 50 points provided for the circle.
Robust geometric inference for a 3-sphere already seems unrealistic, needing perhaps $50^3 = 125,000$ data.
This exponential growth means the very question of robust geometric inference is only meaningful when the intrinsic dimension is low: it is \q{cursed by dimensionality}.

Many methods (some of which are discussed below) can, for example, precisely identify a 20-dimensional manifold as such in a 50-dimensional ambient space.
We will not address this kind of question here, because, we argue, it could never realistically be asked of data with any amount of noise.
The ambient dimension may be very high, but the \textit{intrinsic dimension} will always be low.
The same limitations apply to the Ricci and Riemann curvatures, which only contribute new information in dimensions 3 and 4: the smallest nontrivial use case for Ricci would be a 3-dimensional manifold embedded in 4-dimensional space.
We introduce estimators for these two curvatures, but, given this curse of dimensionality, they are probably not going to be very useful.

\subsection{Is the manifold hypothesis reasonable, and does it matter?}

The theory of diffusion geometry uses Markov diffusions to extend the familiar Riemannian geometry beyond manifolds to more general probability spaces.
Part of the motivation for this was that most of the time in practice, the manifold hypothesis is an unrealistic expectation.
One of the main reasons is that manifolds have a constant integer dimension everywhere, which can hardly be said of most real data (whatever \q{dimension} is taken to mean for data).
Manifolds also cannot have \q{branches}, like three lines meeting at a point or three planes meeting in a line.
However, the methods described here and the Laplacian estimation that underpins them (as well as all the existing \textit{hard neighbourhood} methods discussed) are only defined \textit{locally}, so will still apply to non-manifold data wherever that data looks locally like a manifold.
As such, the manifold hypothesis need only be a metaphor, because the methods discussed here can, informally, be applied more generally.

\subsection{Paper layout}

In Section \ref{sec: geometry background} we review the necessary background in Riemannian geometry, and in Section \ref{sec: computation} we use the Laplacian $\Delta$ to describe estimators for the objects mentioned above.
In Sections \ref{sec: dimension estimation}, \ref{sec: tangent space estimation}, and \ref{sec: curvature estimation} we compute examples and comparative tests of each method.

The code for the methods and examples in this paper is available at \url{https://github.com/Iolo-Jones/ManifoldDiffusionGeometry}.

\section{Geometry background}
\label{sec: geometry background}

In modern differential geometry, manifolds $\M$ are usually defined intrinsically (just in terms of measurements made \textit{within} $\M$) and any embedding of $\M$ into $\R^D$ is viewed as additional structure.
However, in data science, we usually encounter manifolds already embedded in $\R^D$, and can only infer the intrinsic geometry from the \textit{extrinsic} geometry.
In this sense, the extrinsic perspective is more natural here, and we will develop tools in this paper using the extrinsic definitions of geometric objects, which we now review.

\subsection{Tangent spaces and the metric}

Suppose $\M \subseteq \R^D$ is a $d$-dimensional submanifold of $\R^D$ (when $d=2$ we call $\M$ a surface).
A point $p \in \M$ has a tangent space $T_p\M$, corresponding to a $d$-dimensional linear subspace of $\R^D$.
We can equip the tangent space with an inner product $g$ by setting $g(\vec{v},\vec{w}) = \vec{v} \cdot \vec{w}$ for all tangent vectors $\vec{v},\vec{w}\in T_p\M$.
This inner product is called the \textit{induced metric} (or the \textit{first fundamental form} when $\M$ is a surface) and turns $\M$ into an \textit{isometrically embedded} Riemannian submanifold of $\R^D$.

\subsection{Curvature of hypersurfaces}

We first define the curvature of $\M$ in the simpler case that $\M$ is a hypersurface (i.e. $d + 1 = D$) so has, up to sign, a single unit normal $\vec{n}_p$ at each $p\in\M$.
If $f:\M \rightarrow \R$ is a function and $X,Y$ are vector fields, we can define the Hessian $H(f)(X,Y): = X(Y(f)) - (\nabla_X Y)(f)$.
This generalises the matrix of second derivatives
\[
\Big( \frac{\partial^2 f}{\partial x_i \partial x_j} \Big)_{i,j}
\]
for functions on Euclidean space, and measures the \q{curvature} of the function $f$ with respect to $\M$.
If we let $n_p$ denote the function on $\M$ given by $q \mapsto \vec{n}_p \cdot q$, then, near $p$, the manifold $\M$ is locally diffeomorphic to the graph of the function $n_p$, and we can characterise the curvature of $\M$ in terms of the Hessian $H_p(n_p)$.
The Hessian is trilinear in $f,X,Y$, so $H(f)$ is a quadratic form on $T_p\M$ for each $p\in\M$ (it is a $d\times d$ symmetric matrix).


In high dimensions, the geometry is fully characterised by the \textit{Riemann curvature tensor}, which, if $X,Y,Z$ are vector fields, is an operator $Z \mapsto R(X,Y)Z$.
For our purposes, we define it through the \textit{Gauss equation}
\[
g(R(X,Y)Z,W)(p) = H_p(n_p)(X,Z) H_p(n_p)(Y,W) - H_p(n_p)(Y,Z)H_p(n_p)(X,W).
\]
Notice that this equation is unchanged if we instead took the opposite unit normal $-\vec{n}_p$.
This equation relates the intrinsic curvature of $\M$ (measured by $R$) to the curvature of the ambient coordinate functions \textit{as viewed from within $\M$} (measured by $H_p(n_p)$).
If we let $x_i$ be local coordinates near $p$ such that $\nabla_px_i$ are an orthonormal basis for the tangent space $T_p\M$ and set 
\[
R_{ijkl} = g(R(\nabla x_i,\nabla x_j)\nabla x_k,\nabla x_l)(p)
\quad\quad
\alpha_{ij} = H_p(n_p)(\nabla_px_i,\nabla_px_j)
\]
then we get
\[
R_{ijkl} = \alpha_{ik}\alpha_{jl} - \alpha_{jk}\alpha_{il}.
\]

In other words, we can write the Hessian matrix $H_p(n_p)$ in the orthonormal basis $\nabla_px_i$ as $(\alpha_{ij})_{i,j}$.
The simpler Ricci and scalar curvatures are defined using $R$ as
\[
\text{Ric}_{ij} = \sum_{k=1}^d R_{kikj} = \sum_{k=1}^d (\alpha_{kk}\alpha_{ij} - \alpha_{ik}\alpha_{jk})
\]
and
\[
S = \sum_{i = 1}^d \text{Ric}_{ii} = \sum_{i,j = 1}^d R_{ijij} = 2\sum_{i<j} (\alpha_{ii}\alpha_{jj} - \alpha_{ij}^2).
\]

Hypersurfaces \textit{that are orientable} also have a generalised notion of Gaussian curvature, given by $\kappa = \det(H_p(n_p))$.
Note that, when $d=2$, $S = 2(\alpha_{11}\alpha_{22} - \alpha_{12}^2) = 2\kappa$, but in all other dimensions the scalar and Gaussian curvature are not proportional.

We can simplify the above expressions by diagonalising the Hessian $H_p(n_p)$ to get an orthonormal basis for $T_p\M$ of eigenvectors $\nabla_px_i$ with eigenvalues $\lambda_i$, so now $\alpha_{ij} = \lambda_i\delta_{ij}$.
We call the $\lambda_i$ the principal curvatures and the $\nabla_px_i$ the principal curvature directions.
The formula for the Riemann curvature now becomes
\[
R_{ijkl} = 
\begin{cases}
\lambda_i\lambda_j & \text{if } i=k \neq j=l \\
-\lambda_i\lambda_j & \text{if } i=l \neq j=k \\
0 & \text{otherwise}
\end{cases}
\]
and the Ricci curvature becomes
\[
\text{Ric}_{ij} = \sum_{k=1}^d (\lambda_k\lambda_i \delta_{ij} - \lambda_i\lambda_j \delta_{ik}\delta_{jk}) = \big(\lambda_i \sum_{k\neq i} \lambda_k \big) \delta_{ij}.
\]
Notice that this expression is diagonal, so the Ricci curvature has the same eigenvectors as $H_p(n_p)$ (the principal curvature directions).
The scalar curvature is
\[
S = 2\sum_{i<j} \lambda_i\lambda_j.
\]
In the special case of surfaces, we get $\text{Ric} = \lambda_1\lambda_2 I$ and $S = 2\lambda_1\lambda_2$.

\subsection{Curvature of general manifolds}

In the general case, when $\M$ is not a hypersurface, we have a $(D-d)$-dimensional normal space at each point $p \in \M$.
If $\vec{n}_p^\ell$, $\ell = 1,...,D-d$ is an orthonormal basis for the normal space at $p$, and $X,Y$ are vector fields, we define the \textit{second fundamental form} at $p$ to be the vector
\[
\mathrm{I\!I}_p (X,Y) = \sum_{\ell=1}^{D-d} H_p(n_p^\ell)(X,Y) \vec{n}_p^\ell.
\]
The Gauss equation now reads

\begin{equation*}
\begin{split}
g(R(X,Y)Z,W)(p) 
&= \mathrm{I\!I}_p(X,Z) \cdot \mathrm{I\!I}_p(Y,W) - \mathrm{I\!I}_p(Y,Z) \cdot \mathrm{I\!I}_p(X,W) \\
&= \sum_{\ell=1}^{D-d} \big[ H_p(n_p^\ell)(X,Z) H_p(n_p^\ell)(Y,W) - H_p(n_p^\ell)(Y,Z)H_p(n_p^\ell)(X,W) \big]. \\
\end{split}
\end{equation*}

We can therefore apply the same analysis as before to each unit normal $\vec{n}_p^\ell$ separately, and then sum them up.
Writing the Hessian $H_p(n_p^\ell)(\nabla_px_i,\nabla_px_j) = (\alpha^\ell_{ij})_{i,j}$, we obtain the Riemann
\[
R_{ijkl} = \sum_{\ell=1}^{D-d} (\alpha^\ell_{ik}\alpha^\ell_{jl} - \alpha^\ell_{jk}\alpha^\ell_{il}),
\]
Ricci
\[
\text{Ric}_{ij} = \sum_{\ell=1}^{D-d} \sum_{k=1}^d (\alpha^\ell_{kk}\alpha^\ell_{ij} - \alpha^\ell_{ik}\alpha^\ell_{jk}),
\]
and scalar curvatures
\[
S = \sum_{\ell=1}^{D-d} \sum_{i,j=1}^d (\alpha^\ell_{ii}\alpha^\ell_{jj} - (\alpha^\ell_{ij})^2).
\]

These equations hold whatever choice of orthonormal basis $n^\ell_p$ we made.



\section{Computing the dimension, tangent space, and curvature}
\label{sec: computation}

Suppose $\M \subseteq \R^D$ is a $d$-dimensional submanifold of $\R^D$ and $\{p_1,...,p_n\} \subset \M$ is a sample of data from $\M$.
We now outline an approach to computing the dimension $d$, the tangent spaces $T_{p_i}\M$, and the Riemann, Ricci, and scalar curvatures.

The code is available at \url{https://github.com/Iolo-Jones/ManifoldDiffusionGeometry}.

\subsection{Diffusion maps and diffusion geometry}

The geometry of $\M$ is captured by its Riemannian metric $g$, which we can access through the \q{carré du champ} formula
\begin{equation}\label{cdc formula}
g(\nabla f, \nabla h) = \frac{1}{2}\big(f\Delta(h) + h \Delta(f) - \Delta(fh) \big),
\end{equation}
which expresses the Riemannian metric of the vector fields $\nabla f$ and $\nabla h$ in terms of the Laplacian $\Delta$.
Our strategy will be to estimate $\Delta$ from the data, then use the carré du champ formula to compute the first and second fundamental forms.
This approach (called \textit{diffusion geometry}) was developed extensively in \cite{jones2024diffusion} where we used the carré du champ to develop a theory of differential geometry for general probability spaces and data sampled from them; here we apply those methods to the special case of data from manifolds.

Let $X = \{p_1,...,p_n\}$ be our sample of $n$ data.
If $\A = \{f: X \rightarrow \R\}$ is the algebra of functions on $X$, we can identify $\A \cong \R^n$ where the unit vector $e_i \in \R^n$ corresponds to the function $p_j \mapsto \delta_{ij}$.
There is a vast body of work on the problem of estimating $\Delta$ from $X$, and we will use the \textit{diffusion maps} method \cite{COIFMAN20065} which forms a normalised heat kernel matrix from $X$ to estimate $\Delta$ on a compact manifold.
In other words, we compute the $n \times n$ matrix
\[
K_\epsilon(p_i,p_j) = \exp\Big(- \frac{\|p_i - p_j\|^2}{\epsilon} \Big)
\]
for all $i,j$ and bandwidth parameter $\epsilon$, and use it to construct an estimate $\hat{\Delta}_\epsilon = (\text{I} - K_\epsilon)/\epsilon$ for $\Delta$.
In other words, $\hat{\Delta}_\epsilon$ is an $n \times n$ matrix such that, if $f \in C^3(\M)$ and $\hat{f}$ is the vector $\hat{f}_i = f(p_i)$, then, up to rescaling $\hat{\Delta}_\epsilon$ by a constant positive factor, $(\hat{\Delta}_\epsilon\hat{f})_i \approx \Delta(f)(p_i)$.
The use of the heat kernel means that this method is very robust to noise.
We will specifically use \textit{variable bandwidth diffusion kernels} 
\cite{berry2016variable} of the form
\[
K_\epsilon(p_i,p_j) = \exp\Big(- \frac{\|p_i - p_j\|^2}{\epsilon \rho(x)\rho(y)} \Big),
\]
where $\rho$ is an automatically defined bandwidth function, which further improve performance and generalise the method to the non-compact case.
We will use the estimate $\hat{\Delta}_\epsilon$ given in \cite{berry2016variable}, where the authors show the following convergence result.
\\
\begin{theorem}[From Corollary 1, \protect\cite{berry2016variable}]
\label{berry convergence result}
Let $q \in L^1(\M) \cap C^3(\M)$ be a density that is bounded above on $\M$ and let $X$ be sampled independently with distribution $q$. If $f \in L^2(\M,q) \cap C^3(\M)$ is a smooth function, $\hat{f}$ is the vector $\hat{f}_i = f(p_i)$, and $p_i \in X$ is an arbitrary point then, with high probability,
\[
(\hat{\Delta}_\epsilon\hat{f})_i = \Delta(f)(p_i) + \mathcal{O}\Big( \epsilon, \frac{q(p_i)^{1/2 + d/4}}{\sqrt{n}\epsilon^{2 + d/4}}, \frac{\|\nabla_{p_i} f\| q(p_i)^{d^2/2 - 5d/4 + 1} }{\sqrt{n}\epsilon^{1/2 + d/4}} \Big)
\]
up to rescaling $\hat{\Delta}_\epsilon$ by a constant positive factor $c$.
\end{theorem}

Crucially, this estimate is normalised to be independent of the sampling density $q$, so we accurately recover $\Delta$ even when the data are not uniformly sampled from $\M$.
There is a straightforward procedure for automatically selecting $\epsilon$,  which we adopt from \cite{coifman2008graph}, so obtain a parameter-free estimate $\hat{\Delta}$.

This result is notable when the data does not actually lie on a manifold.
Manifolds are locally homeomorphic to Euclidean space everywhere, and the dimension of that Euclidean space must be constant everywhere.
Real data, on the other hand, may only be locally homeomorphic to Euclidean space in some areas, and its dimension might jump around.
The pointwise convergence of Theorem \ref{berry convergence result} ensures that, in this (more realistic) situation, the dimension, tangent space, and curvature estimates defined here will still make sense locally if not globally.

We can use $\hat{\Delta}$ to estimate the carré du champ by plugging it into the formula (\ref{cdc formula}).
If we let $\Gamma(f,h) = g(\nabla f, \nabla h)$ denote the \q{carré du champ operator}, we can estimate
\[
\hat{\Gamma}(\hat{f},\hat{h}) = \frac{1}{2}\big(\hat{f}\hat{\Delta}\hat{h} + \hat{h} \hat{\Delta}\hat{f} - \hat{\Delta}(\hat{f}\hat{h}) \big) \approx \Gamma(f,h).
\]
We can apply Theorem \ref{berry convergence result} to obtain a nearly identical pointwise convergence result for $\hat{\Gamma}$.

\begin{corollary}
\label{cor: cdc convergence}
Under the assumptions of Theorem \ref{berry convergence result}, suppose also that $f,h,fh \in L^2(\M,q) \cap C^3(\M)$ and $f,h$ are bounded.
Then, with high probability,
\[
(\hat{\Gamma}(\hat{f},\hat{h}))_i = \Gamma(f,h)(p_i) + \mathcal{O}\Big( \epsilon, \frac{q(p_i)^{1/2 + d/4}}{\sqrt{n}\epsilon^{2 + d/4}}, \frac{(\|f\|_\infty + \|h\|_\infty + \|\nabla_{p_i} f\| + \|\nabla_{p_i} h\|) q(p_i)^{d^2/2 - 5d/4 + 1} }{\sqrt{n}\epsilon^{1/2 + d/4}} \Big).
\]
\end{corollary}
\begin{proof}
Since $\hat{\Gamma}(\hat{f},\hat{h}) = \frac{1}{2}(\hat{f}\hat{\Delta}\hat{h} + \hat{h} \hat{\Delta}\hat{f} - \hat{\Delta}(\hat{f}\hat{h}))$, the result follows by applying Theorem \ref{berry convergence result} to the three terms $\hat{\Delta}\hat{f}$, $\hat{\Delta}\hat{h}$, and $\hat{\Delta}(\hat{f}\hat{h})$.
The error is of the given order because $f,h$ are bounded and $\| \nabla(fh) \| \leq \|f\|_\infty \|\nabla h\| + \|h\|_\infty \|\nabla f\|$.
\end{proof}

\begin{remark}
Diffusion maps methods, especially in the variable bandwidth case, can result in estimators of the form
\[
\Delta f + b \frac{\nabla f \cdot\nabla q}{q}
\]
where $q$ is the density and $b$ is a constant possibly depending on the (usually unknown) dimension.
In our case, this additional first-order term will not matter, because the carré du champ measures the failure of the Leibniz rule and so any first-order terms will, by definition, cancel out.
\end{remark}

\begin{remark}
While we have followed the standard diffusion maps approach to heat-kernel construction, this computation is possible with arbitrary or abstract graph Laplacians.
\end{remark}

\subsection{Metric and dimension}

We can use the estimate for the carré du champ to compute the dimension $d$ of $\M$ and tangent spaces $T_p\M$.
Notice that the gradients of the ambient coordinate functions $\{ \nabla_p x_i : i = 1,...,D \}$ will span $T_p\M$ at each $p \in \M$.
The Riemannian metric gives an inner product in $T_p\M$, and we can form the following $D \times D$ Gram matrix $G(p)$ of inner products for the gradients of the coordinate functions:
\[
G(p) =
\begin{pmatrix}
g(\nabla_p x_1, \nabla_p x_1) & \dots & g(\nabla_p x_1, \nabla_p x_D) \\
\dots && \dots \\
g(\nabla_p x_D, \nabla_p x_1) & \dots & g(\nabla_p x_D, \nabla_p x_D) \\
\end{pmatrix}
=
\begin{pmatrix}
\Gamma(x_1,x_1) & \dots & \Gamma(x_1,x_D) \\
\dots && \dots \\
\Gamma(x_D,x_1) & \dots & \Gamma(x_D,x_D) \\
\end{pmatrix}
\]

This matrix is symmetric and has rank $d$ at every $p$ (since $T_p\M$ is $d$-dimensional), so has $d$ positive eigenvalues whose eigenvectors form an orthonormal basis for $T_p\M$.
The fact that $\M$ is \textit{isometrically} embedded means that those positive eigenvalues will all be 1.
The other $D-d$ eigenvectors (with eigenvalue 0) form an orthonormal basis for the \q{normal space} of vectors orthogonal to $T_p\M$.

\subsubsection{Tangent space estimation}
\label{Tangent space estimation}

We can use our estimate $\hat{\Gamma}$ for the carré du champ $\Gamma$ to form an estimate $\hat{G}(p_i)$ of $G(p_i)$, i.e. we compute a $D\times D$ matrix $\hat{G}(p_i)$.
If we then diagonalise $\hat{G}(p_i)$, its $D$ eigenvectors will form an orthonormal basis for $\R^D$ and the first $d$ of those should form an orthonormal basis for $T_{p_i}\M$.
The entries of $\hat{G}(p_i)$ converge to their true values as $n \to \infty$ by Corollary \ref{cor: cdc convergence}, so its eigenvalues and eigenvectors also converge appropriately.

This incurs a complexity of $\mathcal{O}(D^3)$ to diagonalise each $G(p_i)$, leading to $\mathcal{O}(nD^3)$ complexity overall, although this computation parallelises trivially over $n$.
If $d$ is known in advance (or at least an upper bound is known), then this cost could be reduced using randomised eigensolvers that approximate just the dominant eigenspaces.

\subsubsection{Normalisation}
\label{Normalisation}

Recall that the \textit{diffusion maps} Laplacian $\hat{\Delta}$ gives an estimate for $\Delta$ up to rescaling by a constant positive factor $c$, which will also lead to $\hat{\Gamma}$ and $\hat{G}$ being rescaled by $c$.
We can correct this using the fact that the first $d$ eigenvalues $\lambda_1^i,..., \lambda_d^i$ of the matrix $G(x_i)$ should be 1, but the eigenvalues $\hat{\lambda}_1^i,..., \hat{\lambda}_d^i$ we estimate from $\hat{G}$ are rescaled to $c$.
So we can estimate $c$ by computing the mean eigenvalue at each point $p_i$ and then taking the median over $i$,
\[
\hat{c} = \text{median} \Big\{ \frac{1}{d} \sum_{j=1}^d \hat{\lambda_j^i} : i = 1,...,n \Big\}.
\]
We then rescale $\hat{\Delta}$ and $\hat{\Gamma}$ by $1/\hat{c}$.

\subsubsection{Dimension estimation}
\label{Dimension estimation}

If we do not know $d$ in advance, we can estimate it from the eigenvalues of the $\hat{G}(p_i)$ given that, if $\hat{\lambda}_\ell^i$ is the $\ell\thupper$ largest eigenvalue of $\hat{G}(p_i)$, we expect
\[
\hat{\lambda}_\ell^i \approx
\begin{cases}
c & \ell\leq d \\
0 & \ell > d
\end{cases}
\]
for all $i$.
We cannot estimate $c$ as described in \ref{Normalisation} without knowing $d$, so will now just assume that $c=1$, which seems perfectly sufficient in practice.
If we form the vector of differences
\[
D_i = (1 - \hat{\lambda}_1^i, \hat{\lambda}_1^i - \hat{\lambda}_2^i, \hat{\lambda}_2^i - \hat{\lambda}_3^i, ..., \hat{\lambda}_{D-1}^i - \hat{\lambda}_D^i, \hat{\lambda}_D^i)
\]
then $D_i$ should be 1 in the $d\thupper$ entry and zero elsewhere, so we can estimate the dimension at the point $p_i$ to be
\begin{equation}\label{def: dim estimate}
\hat{d}(i) = \text{argmax } D_i.
\end{equation}
Notice that the $c=1$ assumption only affects $\hat{d}$ where the data are 0-dimensional.
When the data lies on a manifold, $\hat{d}(i)$ should equal $d$ everywhere.
More generally, given that $\hat{d}(i)$ is derived from the heat diffusion operator, we can think of it as measuring the number of independent directions through which heat can flow along the data through $x_i$.
We can obtain a global dimension estimate $\hat{d}$ from the pointwise one $\hat{d}(i)$ by taking the median over all $i$.

\subsection{Curvature}
\label{section: curvature estimators}

We have now obtained estimates for the tangent bundle, normal bundle, and carré du champ, and can use these to compute the curvature.
To ease the notation, we will describe the computation at a given point $p=p_i$, and drop the index where possible.

The eigenvectors of $\hat{G}(p)$ form an orthonormal basis for $\R^D$: the first $d$ approximate an orthonormal basis $\nabla_px_1,...,\nabla_px_d$ for $T_p\M$, and the last $D-d$ approximate normal vectors $\vec{n}^1,...,\vec{n}^{D-d}$.
To compute the curvature at $p$, we would like to compute the Hessian terms $\alpha^\ell_{ij} = H_p(n^\ell)(\nabla_px_i,\nabla_px_j)$ for each normal vector $n^\ell$ and all pairs of tangent vectors $\nabla_p x_i, \nabla_p x_j$.
We can write $\alpha^\ell_{ij}$ in terms of the carré du champ, via the formula\footnote{This formula is well-known in the theory of Markov diffusion operators: see \cite{bakry2014analysis} for a survey and Proposition 9.1. in \cite{jones2024diffusion} for a derivation.}
\[
H(f)(\nabla h_1, \nabla h_2) = \frac{1}{2} \big( \Gamma(h_1, \Gamma(h_2, f)) + \Gamma(h_2, \Gamma(h_1, f)) - \Gamma(f, \Gamma(h_1, h_2)) \big),
\]
and so use $\hat{\Gamma}$ to compute estimates $\hat{\alpha}^\ell_{ij}$ for $\alpha^\ell_{ij}$.

We then get estimators for the Riemann
\[
\hat{R}_{ijkl} = \sum_{\ell=1}^{D-d} (\hat{\alpha}^\ell_{ik}\hat{\alpha}^\ell_{jl} - \hat{\alpha}^\ell_{jk}\hat{\alpha}^\ell_{il}),
\]
Ricci
\[
\hat{\text{Ric}}_{ij} = \sum_{\ell=1}^{D-d} \sum_{k=1}^d (\hat{\alpha}^\ell_{kk}\hat{\alpha}^\ell_{ij} - \hat{\alpha}^\ell_{ik}\hat{\alpha}^\ell_{jk}),
\]
and scalar curvatures
\[
\hat{S} = \sum_{\ell=1}^{D-d} \sum_{i,j=1}^d (\hat{\alpha}^\ell_{ii}\hat{\alpha}^\ell_{jj} - (\hat{\alpha}^\ell_{ij})^2).
\]

The Hessian estimates $\hat{\alpha}^\ell_{ij}$ involve the squared $\hat{\Delta}^2$, which is not directly covered by Theorem \ref{berry convergence result}.
A suitable convergence result could likely be obtained by similar means (obtaining curvature guarantees such as \cite{aamari2019nonasymptotic}), although that is beyond the scope of this work.

Having diagonalised the Gram matrices $G(p_i)$ in $\mathcal{O}(nD^3)$, the estimates $\hat{\alpha}^\ell_{ij}$ do not require further eigendecomposition, but have a cost of $\mathcal{O}(nd^2(D-d))$, which in the worst case is also $\mathcal{O}(nD^3)$.

\section{Dimension estimation}
\label{sec: dimension estimation}

\subsection{Pointwise dimension}

We defined an estimate $\hat{d}(i)$ at a point $x_i$ in equation (\ref{def: dim estimate}) as the rank of the metric, or, more generally, as the number of independent directions that heat can flow through $x_i$.
This \textit{diffusion dimension} varies across the data and measures the effective dimension of the space at every point.
When the data are drawn from a manifold $\M$ without boundary, $\hat{d}(i)$ should be constant.
When the data lie on manifolds with boundaries and corners, or on the union of several manifolds (which are examples of \textit{stratified spaces}), $\hat{d}(i)$ can vary and encode interesting local geometry, which we see in Figure \ref{fig:pointwise dim}. 

\begin{figure}[h!]
    \centering
    \includegraphics[width=1\linewidth]{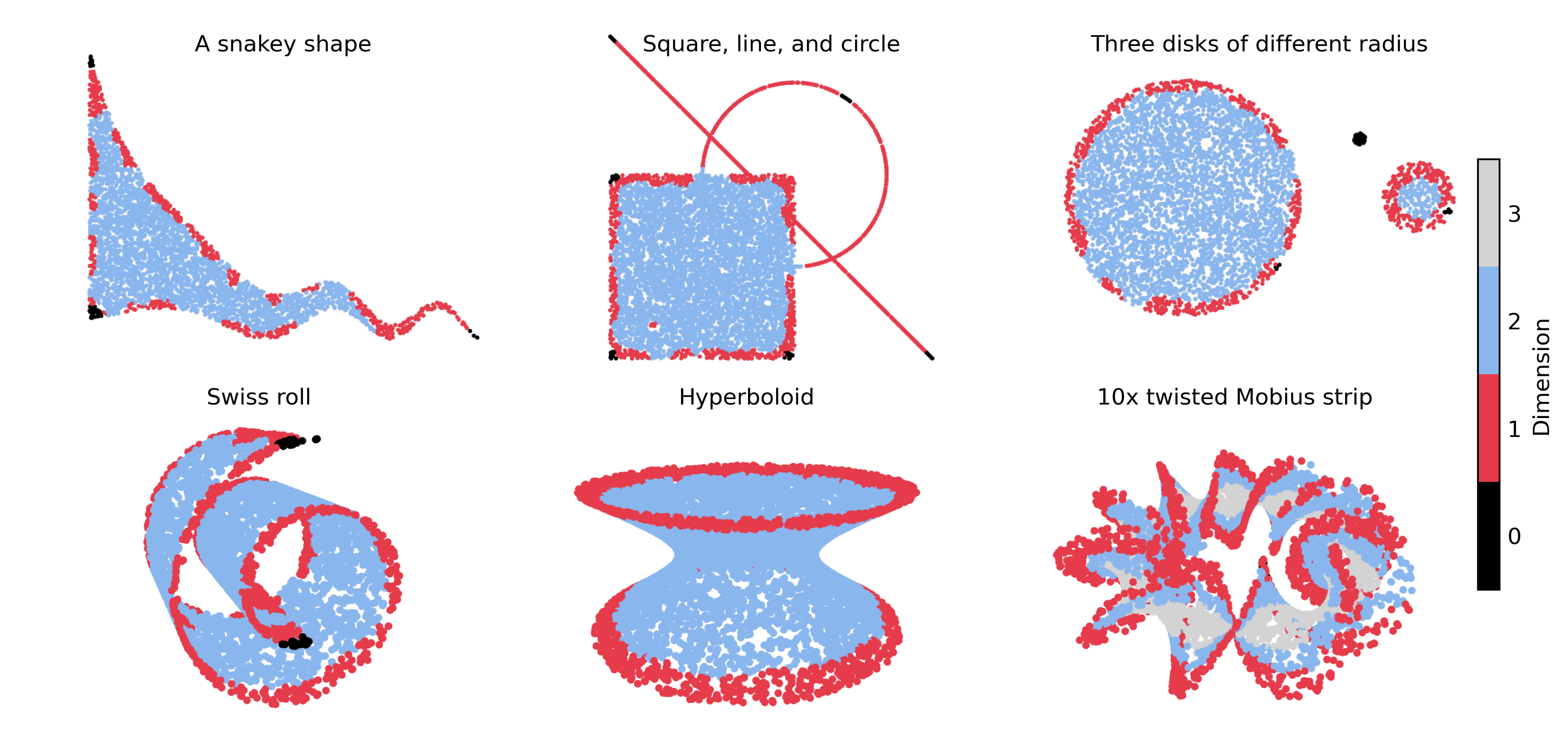}
    \captionsetup{width=0.8\linewidth}
    \caption{\textbf{Pointwise diffusion dimension.} The dimension is estimated at each point for data in 2d (top row) and 3d (bottom row). This process generally identifies 1-dimensional boundaries as 1-dimensional, and hard corners as 0-dimensional. These examples are given without noise for clarity, although this process is very robust to noise.}
    \label{fig:pointwise dim}
\end{figure}

\subsection{Global dimension}

The problem of manifold dimension estimation is very well-studied, and we can define a \textit{global} dimension estimate $\hat{d}$ as the median of $\hat{d}(i)$ across the data.
To test the accuracy of $\hat{d}$, we sample data from the 1,2, and 3-dimensional manifolds in the challenging dimension estimation benchmark proposed in \cite{campadelli2015intrinsic}.
We add to this benchmark the 1, 2, and 3-dimensional hyperspheres and tori, as well as the 2-dimensional hyperboloid, embedded in a variety of ambient dimensions\footnote{The results on the benchmark of \cite{campadelli2015intrinsic} alone are in the Appendix, and are qualitatively very similar.}.
This results in 12 test manifolds, including deliberately difficult examples with high curvature, such as the 10-times twisted Möbius strip in Figure \ref{fig:pointwise dim}, and tightly-wound helices.
We focus on at most 3-dimensional data for the reasons outlined in the introduction.
For each manifold, we sample $n$ data where either $n = n_{small}$ or $n = n_{large}$, where $n_{small}$ and $n_{large}$ vary with dimension\footnote{When $d=1$, $n_{small} = 600$ and $n_{large} = 1200$, when $d=2$, $n_{small} = 1200$ and $n_{large} = 2400$, and when $d=3$, $n_{small} = 2400$ and $n_{large} = 4800$.}.
We also add normally distributed noise with standard deviation 0, $0.5 \sigma_{max}$, or $\sigma_{max}$, where $\sigma_{max}$ is defined separately for each manifold in the benchmark as the largest value such that at least one method can attain at least $50\%$ accuracy with that much noise.

We test the diffusion geometry estimate $\hat{d}$ on this challenging data, against eleven standard and state-of-the-art methods, 
Correlation Dimension \cite{grassberger1983measuring}, 
Dimensionality from Angle and Norm Concentration (DANCo) \cite{ceruti2012danco}, 
Expected Simplex Skewness (ESS) \cite{johnsson2014low}, 
Fisher Separability \cite{albergante2019estimating}, 
local principal component analysis (LPCA) \cite{cangelosi2007component, fan2010intrinsic, fukunaga1971algorithm},
Manifold-Adaptive Dimension Estimation (MADA) \cite{farahmand2007manifold}, 
Minimum Neighbor Distance - Maximum Likelihood (MiND-ML) \cite{rozza2012novel}, 
Maximum Likelihood Estimation (MLE) \cite{haro2008translated, hill1975simple, levina2004maximum}, 
Method of Moments (MOM) \cite{amsaleg2018extreme}, 
Tight Local intrinsic dimensionality Estimator (TLE) \cite{amsaleg2019intrinsic}, 
and Two Nearest Neighbours (TwoNN) \cite{facco2017estimating}.
Some of these methods return non-integer dimension estimates, and in these cases we round to the nearest integer.
Most of these methods also have parameters to be set, but, to fairly test their real-world performance, these parameters are all set to their default values, as implemented in \cite{bac2021scikit}.
We measure the average accuracy of each method across 20 samples of each of the 12 manifolds in the benchmark, for the different numbers of data $n$ and levels of noise $\sigma$, and record the results in Table \ref{tab:benchmark-comparison}.

\setlength{\tabcolsep}{3pt}
\begin{table}[h!]
\centering
\resizebox{\textwidth}{!}{%
\begin{tabular}{l|ccc|ccc}
&& $n = n_{small}$ &&& $n = n_{large}$ & \\
\textbf{Method} & $\sigma = 0$ & $\sigma = 0.5 \sigma_{max}$ & $\sigma = \sigma_{max}$ & $\sigma = 0$ & $\sigma = 0.5 \sigma_{max}$ & $\sigma = \sigma_{max}$ \\
\midrule
Correlation Dimension \cite{grassberger1983measuring}
& 
$\color[HTML]{e53b4a}\mathbf{100.0\pm0.0}$ & $\color[HTML]{e53b4a}\mathbf{87.1\pm6.6}$ & $21.2\pm4.1$ & $91.7\pm0.0$ & $35.8\pm7.5$ & $5.8\pm3.8$ \\
MADA \cite{farahmand2007manifold}
& $91.7\pm0.0$ & $32.9\pm11.5$ & $0.0\pm0.0$ & $91.7\pm0.0$ & $8.3\pm0.0$ & $0.0\pm0.0$ \\
LPCA \cite{cangelosi2007component, fan2010intrinsic, fukunaga1971algorithm}
& $8.3\pm0.0$ & $8.3\pm0.0$ & $0.0\pm0.0$ & $8.3\pm0.0$ & $8.3\pm0.0$ & $0.0\pm0.0$ \\
MLE \cite{haro2008translated, hill1975simple, levina2004maximum}
& $91.7\pm0.0$ & $47.1\pm6.3$ & $0.0\pm0.0$ & $91.7\pm0.0$ & $17.1\pm1.8$ & $0.0\pm0.0$ \\
MiND-ML \cite{rozza2012novel}
& $\color[HTML]{166dde}\mathbf{92.1\pm1.8}$ & $18.3\pm3.3$ & $0.0\pm0.0$ & $\color[HTML]{e53b4a}\mathbf{100.0\pm0.0}$  & $17.1\pm1.8$ & $0.4\pm1.8$ \\
DANCo \cite{ceruti2012danco}
& $91.7\pm0.0$ & $0.0\pm0.0$ & $0.0\pm0.0$ & $\color[HTML]{166dde}\mathbf{98.8\pm3.0}$ & $0.4\pm1.8$ & $0.0\pm0.0$ \\
TwoNN \cite{facco2017estimating}
& $\color[HTML]{e53b4a}\mathbf{100.0\pm0.0}$ & $11.2\pm5.6$ & $7.9\pm1.8$ & $\color[HTML]{e53b4a}\mathbf{100.0\pm0.0}$  & $8.3\pm0.0$ & $8.3\pm0.0$ \\
MOM \cite{amsaleg2018extreme}
& $75.0\pm0.0$ & $66.7\pm0.0$ & $\color[HTML]{166dde}\mathbf{50.0\pm0.0}$ & $83.3\pm0.0$ & $\color[HTML]{166dde}\mathbf{73.3\pm3.3}$ & $\color[HTML]{166dde}\mathbf{20.8\pm4.2}$ \\
Fisher Separability \cite{albergante2019estimating}
& $16.7\pm0.0$ & $16.7\pm0.0$ & $15.4\pm3.0$ & $16.7\pm0.0$ & $16.7\pm0.0$ & $15.0\pm3.3$ \\
TLE \cite{amsaleg2019intrinsic}
& $91.7\pm0.0$ & $0.0\pm0.0$ & $0.0\pm0.0$ & $91.7\pm0.0$ & $0.0\pm0.0$ & $0.0\pm0.0$ \\
\midrule
Diffusion Geometry  & 
$88.8\pm4.0$ & $\color[HTML]{166dde}\mathbf{83.3\pm0.0}$ & $\color[HTML]{e53b4a}\mathbf{55.0\pm7.7}$ & $90.8\pm2.5$ & $\color[HTML]{e53b4a}\mathbf{79.6\pm4.1}$ & $\color[HTML]{e53b4a}\mathbf{58.8\pm11.0}$ \\
\end{tabular}
}

\captionsetup{width=0.8\linewidth}
\caption{\protect\textbf{Manifold dimension estimation.} 
Average accuracies and standard deviations (\%) from 20 runs over 12 benchmark manifolds with dimensions 1, 2, and 3.
The standard deviations are computed for each manifold and then averaged overall, so 0 means that the method returned the same value for each manifold on every run.
We sample $n$ data randomly for a small and large value of $n$, and for zero, medium, and large amounts of noise.
The {\color[HTML]{e53b4a}\textbf{red}} and {\color[HTML]{166dde}\textbf{blue}} numbers indicate the {\color[HTML]{e53b4a}\textbf{best}} and {\color[HTML]{166dde}\textbf{second-best}} accuracies in each column.}
\label{tab:benchmark-comparison}
\end{table}

The many methods tested here all respond to noise and density in different ways: some achieve outstanding accuracy on clean data but have virtually no robustness to noise (MADA, MLE, MiND-ML, DANCo, TwoNN, TLE), and many performed especially well on one or two particular examples in the benchmark.
Overall, the best-performing methods across the test are clearly correlation dimension, method of moments (MOM), and diffusion geometry.
Correlation dimension is highly accurate on clean data and also reasonably robust, but only at the lower density $n_{small}$.
MOM is more robust, and deals better with the change in density, but is overall less accurate.
Diffusion geometry, while not the most accurate method for perfect, noiseless data, is by far the most robust to noise and density, and achieves the highest accuracy of any model in 3 of the 4 noisy examples. 

Low-dimensional data are often encountered in a very high-dimensional ambient space, but we find this does not significantly affect the performance or robustness of these methods.
Running the same test on the same data but embedded in a range of higher ambient dimensions does not change the results.

\section{Tangent space estimation}
\label{sec: tangent space estimation}

In Subsection \ref{Tangent space estimation}, we used the leading eigenvectors of the metric as an approximate basis for the tangent space $T_{x_i}\M$ of a manifold $\M$.
When the data lie on \q{manifold-like} objects such as the union of manifolds, this \q{tangent space} estimate encodes something like the directions of strongest diffusion through the space.
We compute the 1 and 2-dimensional tangent spaces of noisy data in Figure \ref{fig:tangent spaces general}.

\begin{figure}[!h]
    \centering
    \includegraphics[width=1\linewidth]{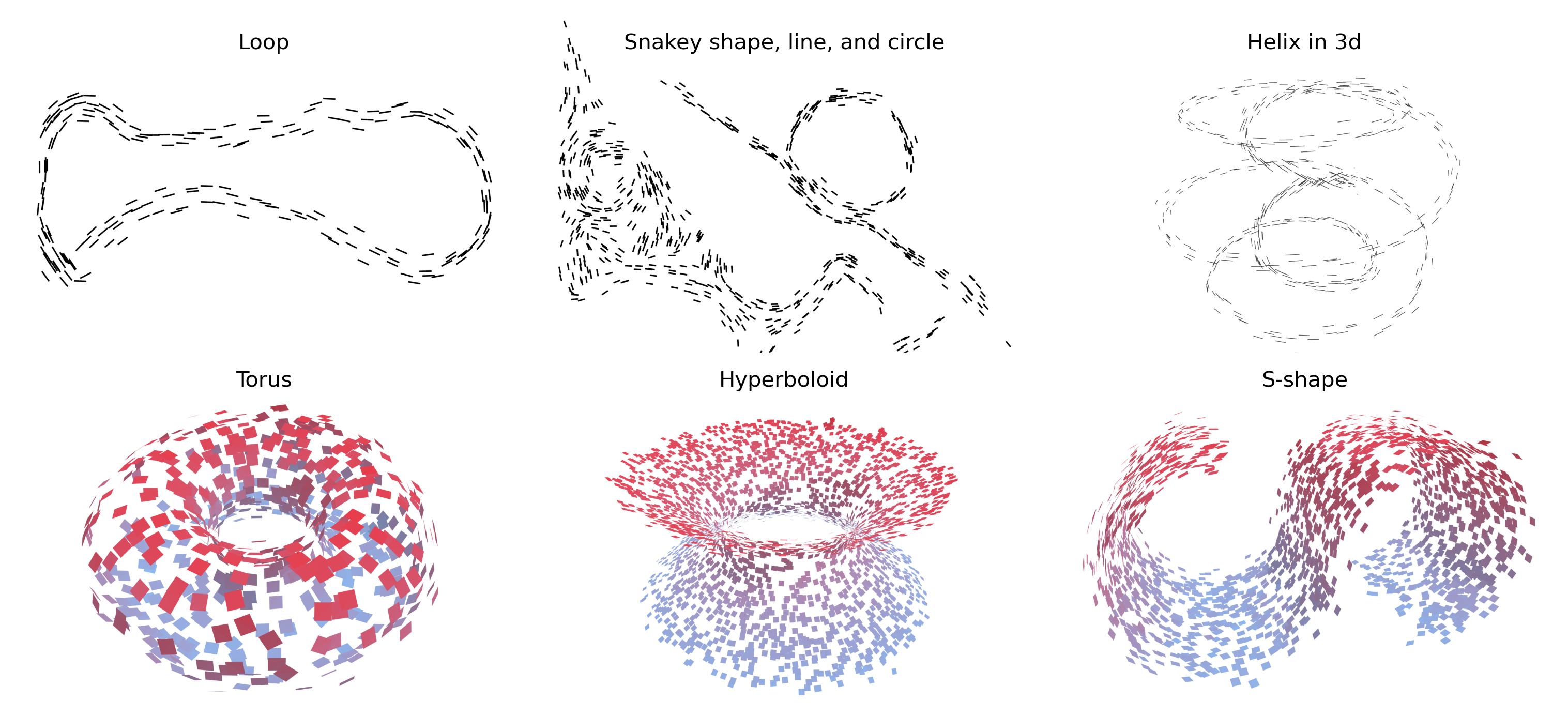}
    \captionsetup{width=0.8\linewidth}
    \caption{\textbf{Tangent spaces of data.} We compute the 1d (top row) and 2d (bottom row) tangent space for each point.
    When the data are from a manifold, we robustly recover the tangent bundle even with large amounts of noise.
    When the data is not a manifold (top centre example), these \q{diffusion tangents} measure the direction of greatest heat flow along the object.}
    \label{fig:tangent spaces general}
\end{figure}

The standard technique for this task is local principal component analysis (LPCA) \cite{aamari2019nonasymptotic}, which, for some given parameter $k$, computes the covariance matrix of the $k$-nearest neighbours to each point $x_i$, and defines the tangent space to be its leading $d$ eigenvectors.
In other words, it performs principal component analysis in a neighbourhood of each point, so follows the \textit{hard neighbourhood paradigm}.
LPCA can be effective in many situations but depends sensitively on the choice of neighbourhood size $k$, whose optimal value depends on the amount of noise, density of the sample, and geometry of the underlying manifold.
When $k$ is small, LPCA can accurately compute the tangent space for clean data but is extremely sensitive to noise.
When $k$ is large, LPCA becomes much more noise-robust but is only accurate on very dense samples.
This trade-off is inherent to \textit{hard neighbourhood} methods.

Conversely, the diffusion geometry tangent space estimator does not need parameter selection and is highly robust to both noise and density simultaneously.
To compare the two methods, we sample data from the torus in Figure \ref{fig:tangent spaces general} with different densities and noise levels.
We estimate the tangent space with diffusion geometry and LPCA with two neighbourhood parameters: $k=5$ and $k=100$.
To test the accuracy of an estimate at point $x_i$, we compute the normal vector to the estimated tangent plane, and compare it with the correct normal to the torus at $x_i$ (or at the point on the torus closest to $x_i$, when the data contains noise).
We then compute the angle between these two normal vectors and average it over the data: when the average is 0\textdegree\ the estimator is perfectly accurate, and when the average is over 45\textdegree\ the tangent spaces are essentially random.
We visualise the results in Figure \ref{fig:LPCA comparison}.

\begin{figure}[!h]
    \centering
    \includegraphics[width=\linewidth]{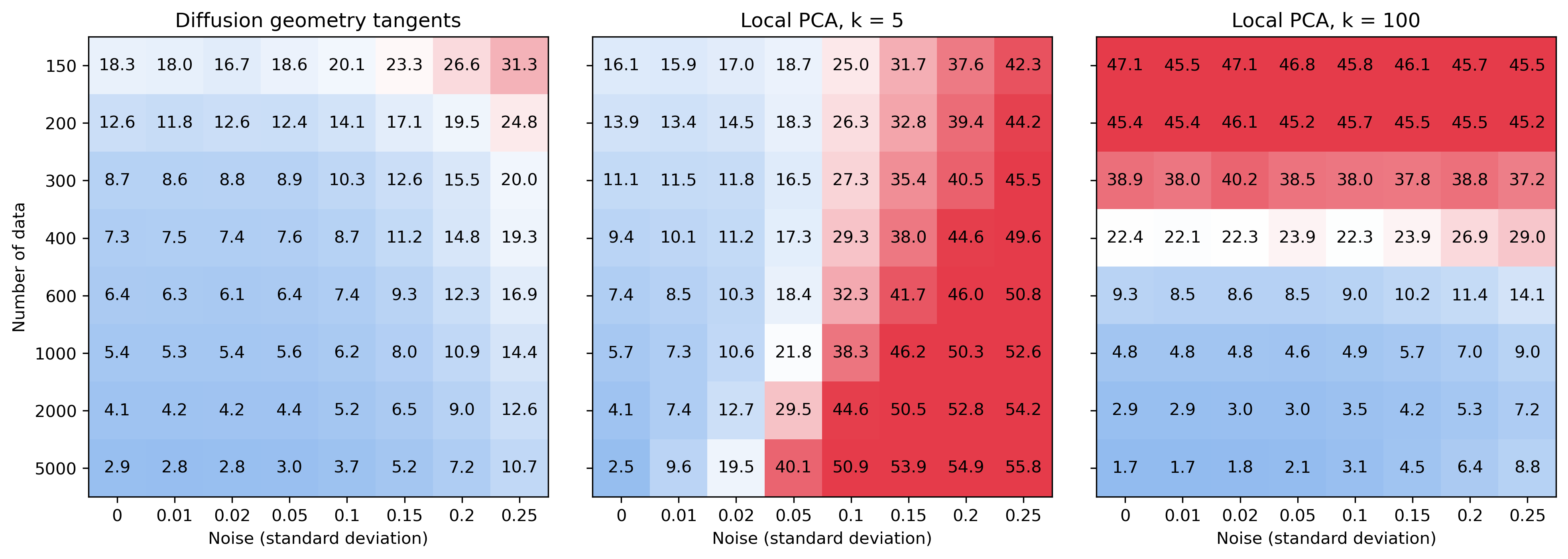}
    \captionsetup{width=0.8\linewidth}
    \caption{\textbf{Diffusion Geometry vs Local PCA.}
    We compute the tangent space for a torus in $\R^3$ with different sampling densities and noise levels, using tangent diffusion and LPCA.
    We measure accuracy at a point by finding the \q{error angle} between the normal vector to the computed tangent space and the true normal at that point.
    The grids contain the average error angle in degrees, averaged over 10 runs: 0\textdegree\ means perfect accuracy over the whole torus and over 45\textdegree\ means the tangent spaces are random.
    LPCA is computed for $k=5$ and $k=100$ nearest neighbours: both values perform well in places but no fixed value of $k$ is always good.
    Diffusion geometry is comparable to or better than LPCA and does not need parameter selection.}
    \label{fig:LPCA comparison}
\end{figure}

We find that, while LPCA can always attain very high accuracy with near-perfect data, no fixed value of $k$ is robust to both noise and density.
Conversely, the diffusion geometry tangents attain comparable accuracy for high-density and low-noise data but can maintain good accuracy even when the data is low-density and highly noisy.
In particular, diffusion geometry outperforms LPCA with either parameter when $n < 1000$ and especially when $\sigma > 0.02$.
The torus data in Figure \ref{fig:tangent spaces general} has $n=600$ and $\sigma = 0.1$ so is right on this boundary: it is clearly a torus to the human eye but is hard for LPCA.
For geometric analysis of real data, it is especially important to perform well on exactly this sort of low-quality, sparse and noisy data.

\section{Curvature estimation}
\label{sec: curvature estimation}

\subsection{Scalar curvature}

In Subsection \ref{section: curvature estimators}, we introduced estimators for the Riemann, Ricci, and scalar curvatures using the second fundamental form defined by the Hessian.
These give us a powerful description of the local geometry of manifolds.
We give several examples of the scalar curvature of a surface in Figure \ref{fig:scalar general examples}.

\begin{figure}[!h]
    \centering
    \includegraphics[width=1\linewidth]{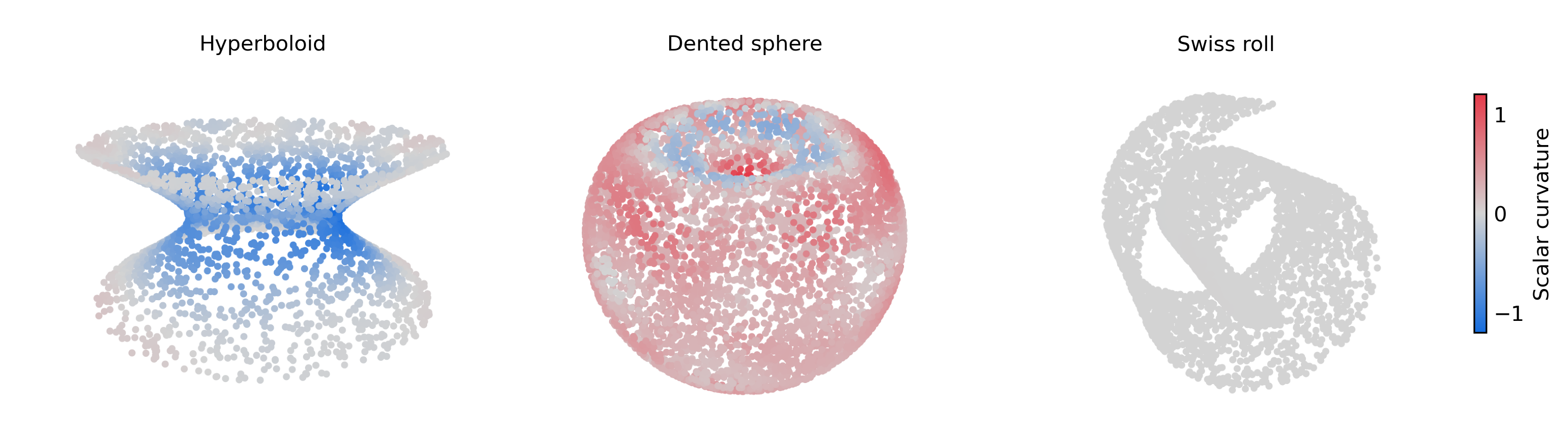}
    \captionsetup{width=0.8\linewidth}
    \caption{\textbf{Scalar curvature of surfaces.} 
    We compute the scalar curvature with diffusion geometry.
    When the surface is locally spherical, the curvature is positive (e.g. on most of the dented sphere).
    When the surface is locally hyperbolic the curvature is negative (e.g. the hyperboloid, which is negatively curved everywhere, and the \q{rim} of the dent on the sphere).
    Zero curvature means the surface looks like Euclidean space, such as on the Swiss roll, which is a wrapped but, crucially, not deformed rectangle.}
    \label{fig:scalar general examples}
\end{figure}

We test our estimate on challenging, low-quality data by sampling points from a torus in $\R^3$

\[
\{\big((R + r \cos(\phi))\cos(\theta), (R + r \cos(\phi))\sin(\theta), r \sin(\theta)\big) : \phi, \theta \in [0, 2\pi) \}
\]

(we choose $r=1$ and $R=2$) and compare the estimated scalar curvature to the correct value
\begin{equation}
\label{eq: torus scalar ground truth}
S(\theta, \phi) = \frac{2 \cos(\phi)}{r(R + r\cos(\phi))}.
\end{equation}
Notice that $S$ depends only on the internal angle $\phi$, so in Figure \ref{fig:scalar torus ground truth} we plot our estimate against \q{ground truth} as functions of $\phi$.
Even with noisy and low-density data, our estimator is very accurate.

\begin{figure}[!h]
    \centering
    \includegraphics[width=1\linewidth]{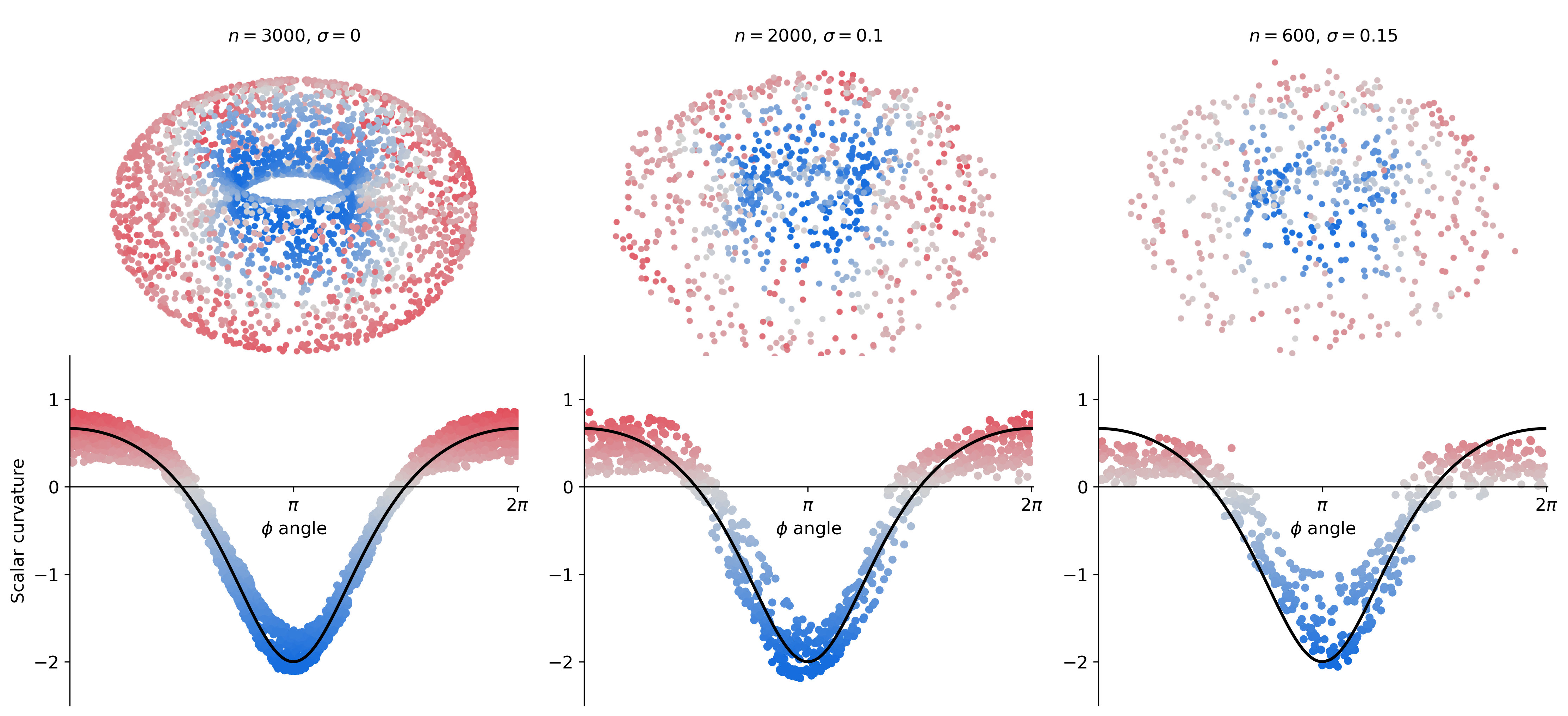}
    \captionsetup{width=0.8\linewidth}
    \caption{\textbf{Diffusion scalar curvature vs ground truth.} We compute scalar curvature on a torus using diffusion geometry (top row) and compare it to the true scalar curvature (bottom row).
    The true curvature (\ref{eq: torus scalar ground truth}) depends only on the internal angle $\phi$ of the torus, and is plotted in black alongside the diffusion geometry estimate.
    When the data contain noise, the estimated curvature at $x_i$ is plotted against the $\phi$ value of the point on the torus closest to $x_i$.}
    \label{fig:scalar torus ground truth}
\end{figure}

There have been, to our knowledge, only two previous attempts to estimate the scalar curvature on a general manifold.
Sritharan, Wang and Hormoz (SWH) \cite{sritharan2021computing} estimated the scalar curvature using a similar extrinsic approach to the one described above, but used local linear and quadratic regression to estimate the tangent spaces and Hessian.
More recently, Hickok and Blumberg (HB) \cite{hickok2023intrinsic} presented a method based on the volume comparison properties of the scalar curvature.
There are also several surface-specific methods for computing scalar curvature \cite{asao2021curvature, yang2007direct, merigot2010voronoi, khameneifar2019curvature}, which generally follow a similar approach to the SWH curvature, and, in our experiments, do not perform better than it, so are not presented here.
There are several other statistics for the curvature of data that are not explicitly related to the Riemannian geometric curvature, such as \q{diffusion curvature} \cite{bhaskar2022diffusion} which also uses diffusion.

Since the SWH curvature also follows the \textit{hard neighbourhood paradigm} of estimating differential objects through local regression in a neighbourhood, it comes with a parameter that controls neighbourhood size.
SWH uses a \q{standard error} parameter that lets the user control the target precision of the method: smaller values are more precise but less robust.
Conversely, the diffusion geometry scalar curvature is not parameterised and does not have to make the \textit{hard neighbourhood} trade-off between accuracy and robustness.
The HB approach is intrinsic, so does not need an embedding of the data, but, in our experiments, is significantly less accurate than the other two methods.

We compare the diffusion geometry and SWH scalar curvatures by the same experiment as in Figure \ref{fig:LPCA comparison}: sampling data from a torus and comparing the computed estimates to ground truth.
We use two different values of standard error parameter, 0.02 and 0.1, in the SWH method.
We average the results over 10 runs and present them in Figure \ref{fig:scalar curvature comparison}.

\begin{figure}[!h]
    \centering
    \includegraphics[width=\linewidth]{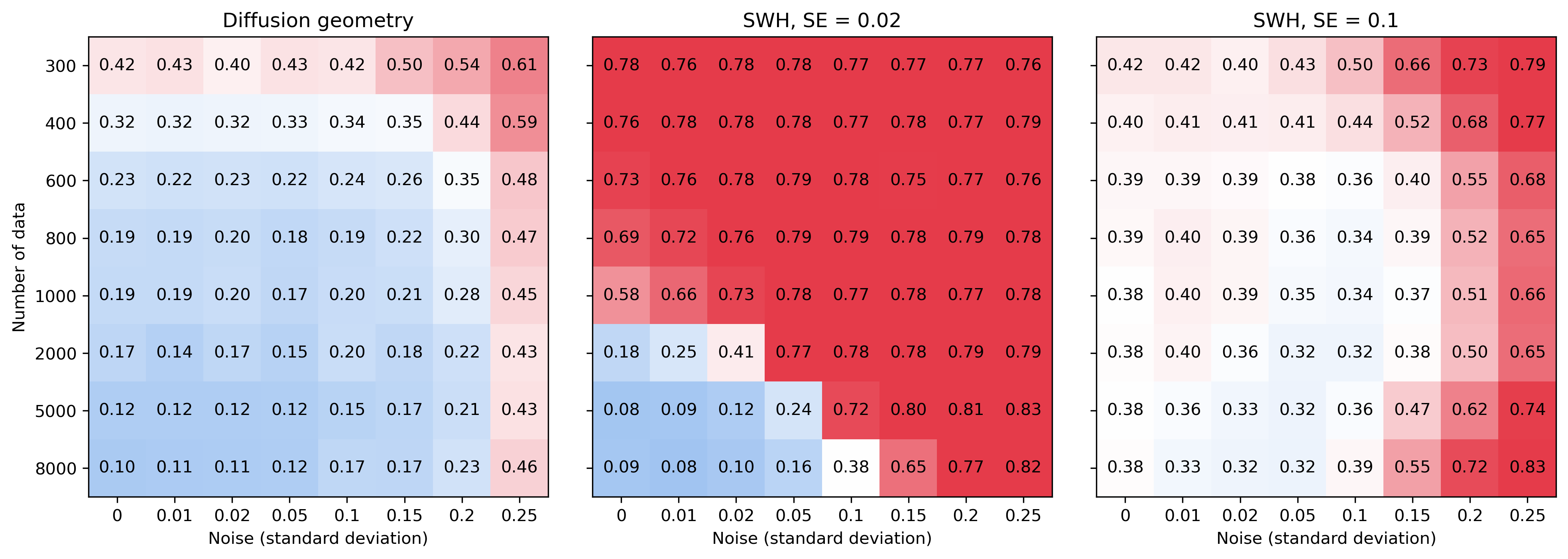}
    \captionsetup{width=0.8\linewidth}
    \caption{\textbf{Diffusion geometry vs Sritharan/Wang/Hormoz scalar curvature.}
    We compute the scalar curvature of a torus in $\R^3$ with different sampling densities and noise levels, using diffusion geometry and the SWH method.
    We compare the estimates to the true scalar curvature (\ref{eq: torus scalar ground truth}) and record the mean absolute error averaged over 10 runs.
    The SWH curvature is computed for the standard error parameters 0.02 and 0.1: this choice presents a trade-off between accuracy and robustness.
    Diffusion geometry does not need parameter selection and attains comparable performance on perfect data, but deals significantly better with noise and sparsity.}
    \label{fig:scalar curvature comparison}
\end{figure}

Just like LPCA, the SWH curvature attains very high accuracy with near-perfect data, but there is a significant trade-off between accuracy and robustness to noise and density.
Conversely, the diffusion geometry scalar curvature has comparable accuracy for high-density and low-noise data but still performs very well when the data is sparse and noisy.
The two noisy examples in Figure \ref{fig:scalar torus ground truth} are much too sparse and noisy for SWH, but can be computed very accurately with diffusion geometry.

Even though the diffusion geometry method presented here is manifold-specific, these strong statistical results suggest that it could provide meaningful curvature analysis on real-world data.

\subsection{Ricci and Riemann curvature}

In dimension 2, the Ricci and Riemann curvatures are just multiples of the scalar curvature.
This relationship is respected by our methods, and so the above results demonstrate the robust estimation of all three curvature tensors in dimension less than 3.

The Ricci curvature starts to differ from the scalar curvature in dimension 3.
We tested the Ricci estimate given in Subsection \ref{section: curvature estimators} on several examples, but obtaining a high enough density to really observe its deviation from the scalar curvature requires $n$ to be huge.
Namely, while the estimate for scalar curvature remained accurate, the second component of Ricci curvature was often small or inaccurate for moderate sizes of data.
We could not find any interesting examples where $n < 50,000$, which is too large to be practical.

There are, to our knowledge, fewer directly comparable existing methods for Ricci curvature.
The papers \cite{ache2019ricci} and \cite{ache2022approximating} describe an estimator for it, but do not test it on data.
There is a related body of work on graph curvature, of which the Ollivier Ricci curvature \cite{ollivier2007ricci} has been applied to point clouds by turning them into graphs.
This defines a Ricci curvature measure on each edge of the graph, which converges pointwise \cite{PhysRevResearch, trillos2023continuum}.
However, like the scalar curvature method of \cite{hickok2023intrinsic} mentioned above, these graph methods use only the pairwise distances between the data and not the original ambient coordinates.
These purely intrinsic methods seem to require much larger and cleaner datasets to obtain good results than the extrinsic methods presented here.

The Riemann curvature only differs from the Ricci in dimension 4, where obtaining a sufficiently dense sample to notice the difference with these methods seems infeasible.

\section{Limitations of non-local data analysis}

The methods presented here follow the standard \q{non-local} paradigm for geometric and topological data analysis.
That is, they approximate \textit{local} objects like differential operators (which depend only on a point $x \in \M$) with \textit{non-local} kernels, which take into account the values in a neighbourhood of $x$.
This approach is fundamentally cursed by dimensionality, because its success depends on having a sufficiently dense sample of a neighbourhood, which is infeasible in high dimensions (see discussion in \ref{sub: apology}).
To estimate first-order information, like tangents and dimensions, these methods require a dense sample of a large enough neighbourhood for the kernel to \q{see} the direction of the tangent space.
The effect on second-order information like curvature is more glaring, where the kernel needs to additionally observe the relationships between all of those tangent spaces.
The higher-order Ricci and Riemann curvature tensors, which further stratify the different directions of curvature, need yet more data to be accurate.

While the results in this paper demonstrate successful estimation of scalar curvature in low intrinsic dimensions, and could still be improved with further engineering, it seems unlikely that the ceiling on the ability of non-local methods is that much higher.
Using the intuition that a ball of radius $r$ in dimension $d$ has volume $\ord(r^d)$, we see that \textit{small} balls in \textit{high} dimensions have vanishingly small volumes, and so the neighbourhoods of our data will become exponentially sparse for large $d$.
As such, genuinely high-dimensional data analysis may only be possible if the non-local approach is replaced.

Surely, the most promising candidate for a new paradigm would be to use neural networks that contain the right inductive biases to \q{connect} local data points without having to explicitly measure the distances between them.
This kind of approach has already been applied to dimension estimation in \cite{stanczuk2022your}, and may prove useful for computing other geometric quantities like curvature, although we will leave that discussion for future work.

\section{Conclusions}

We have introduced a range of novel computational geometry tools for data on manifolds, based on diffusion geometry.
While these new methods attain comparable results to the existing state-of-the-art on low-noise, dense data, they significantly outperform existing methods in the presence of noise or sparsity.

We have consistently found that methods that follow the \textit{hard neighbourhood paradigm} must trade off precision against robustness, and come with a user-defined parameter to control this.
Even the optimal choice of this parameter is usually not enough to attain high accuracy on very sparse and noisy data.
Conversely, diffusion geometry is not a \textit{hard neighbourhood} method, so enjoys both accuracy and robustness simultaneously while also being parameter-free.

Important future research questions include: can we define and measure curvature on non-manifold data, and can local dimensionality and curvature measures work effectively as features for geometric machine learning?

\section*{Acknowledgements}

I want to thank Jeff Giansiracusa and Yue Ren for their generous comments and feedback.
This work was carried out as part of EPSRC grant EP/Y028872/1, \textit{Mathematical Foundations of Intelligence: An \qq{Erlangen Programme} for AI}.


\section*{Appendix: other dimension benchmark results}

The benchmark used for the results in Table \ref{tab:benchmark-comparison} included six manifolds from the benchmark proposed in \cite{campadelli2015intrinsic}, as well as six additional ones.
For completeness, we include the results obtained \textit{only on the benchmark data from} \cite{campadelli2015intrinsic} in Table \ref{tab:benchmark-comparison limited}.
There are seven manifolds in the benchmark from \cite{campadelli2015intrinsic} of dimension at most 3, although one of these, the \q{M13b Spiral}, was too challenging for any of the methods considered to be accurate with any amount of noise, so is excluded.

\setlength{\tabcolsep}{3pt}
\begin{table}[h!]
\centering
\resizebox{\textwidth}{!}{%
\begin{tabular}{l|ccc|ccc}
&& $n = n_{small}$ &&& $n = n_{large}$ & \\
\textbf{Method} & $\sigma = 0$ & $\sigma = 0.5 \sigma_{max}$ & $\sigma = \sigma_{max}$ & $\sigma = 0$ & $\sigma = 0.5 \sigma_{max}$ & $\sigma = \sigma_{max}$ \\
\midrule
Correlation Dimension \cite{grassberger1983measuring}
& 
$\color[HTML]{e53b4a}\mathbf{100.0\pm0.0}$ & $\color[HTML]{e53b4a}\mathbf{75.8\pm8.3}$ & $16.7\pm0.0$ & $83.3\pm0.0$ & $32.5\pm3.6$ & $0.0\pm0.0$ \\
MADA \cite{farahmand2007manifold}
& $83.3\pm0.0$ & $20.8\pm7.2$ & $0.0\pm0.0$ & $83.3\pm0.0$ & $0.0\pm0.0$ & $0.0\pm0.0$ \\
LPCA \cite{cangelosi2007component, fan2010intrinsic, fukunaga1971algorithm}
& $16.7\pm0.0$ & $16.7\pm0.0$ & $0.0\pm0.0$ & $16.7\pm0.0$ & $16.7\pm0.0$ & $0.0\pm0.0$ \\
MLE \cite{haro2008translated, hill1975simple, levina2004maximum}
& $83.3\pm0.0$ & $33.3\pm0.0$ & $0.0\pm0.0$ & $83.3\pm0.0$ & $0.8\pm3.6$ & $0.0\pm0.0$ \\
MiND-ML \cite{rozza2012novel}
& $\color[HTML]{166dde}\mathbf{84.2\pm3.6}$ & $3.3\pm6.7$ & $0.0\pm0.0$ & $\color[HTML]{e53b4a}\mathbf{100.0\pm0.0}$ & $16.7\pm0.0$ & $0.8\pm3.6$ \\
DANCo \cite{ceruti2012danco}
& $83.3\pm0.0$ & $0.0\pm0.0$ & $0.0\pm0.0$ & $\color[HTML]{166dde}\mathbf{97.5\pm6.0}$ & $0.8\pm3.6$ & $0.0\pm0.0$ \\
TwoNN \cite{facco2017estimating}
& $\color[HTML]{e53b4a}\mathbf{100.0\pm0.0}$ & $16.7\pm0.0$ & $15.8\pm3.6$ & $\color[HTML]{e53b4a}\mathbf{100.0\pm0.0}$ & $16.7\pm0.0$ & $16.7\pm0.0$ \\
MOM \cite{amsaleg2018extreme}
& $50.0\pm0.0$ & $33.3\pm0.0$ & $16.7\pm0.0$ & $66.7\pm0.0$ & $\color[HTML]{166dde}\mathbf{46.7\pm6.7}$ & $0.0\pm0.0$ \\
Fisher Separability \cite{albergante2019estimating}
& $33.3\pm0.0$ & $33.3\pm0.0$ & $\color[HTML]{166dde}\mathbf{30.8\pm6.0}$ & $33.3\pm0.0$ & $33.3\pm0.0$ & $\color[HTML]{166dde}\mathbf{30.0\pm6.7}$ \\
TLE \cite{amsaleg2019intrinsic}
& $83.3\pm0.0$ & $0.0\pm0.0$ & $0.0\pm0.0$ & $83.3\pm0.0$ & $0.0\pm0.0$ & $0.0\pm0.0$ \\
\midrule
Diffusion Geometry  & 
$77.5\pm7.9$ & $\color[HTML]{166dde}\mathbf{66.7\pm0.0}$ & $\color[HTML]{e53b4a}\mathbf{42.5\pm11.8}$ & $81.7\pm5.0$ & $\color[HTML]{e53b4a}\mathbf{66.7\pm0.0}$ & $\color[HTML]{e53b4a}\mathbf{36.7\pm16.0}$ \\
\end{tabular}
}

\captionsetup{width=0.8\linewidth}
\caption{\protect\textbf{Manifold dimension estimation on the benchmark data from \protect\cite{campadelli2015intrinsic} only.}
Average accuracies and standard deviations (\%) from 20 runs over 12 benchmark manifolds with dimensions 1, 2, and 3.
The standard deviations are computed for each manifold and then averaged overall, so 0 means that the method returned the same value for each manifold on every run.
We sample $n$ data randomly for a small and large value of $n$, and for zero, medium, and large amounts of noise.
The {\color[HTML]{e53b4a}\textbf{red}} and {\color[HTML]{166dde}\textbf{blue}} numbers indicate the {\color[HTML]{e53b4a}\textbf{best}} and {\color[HTML]{166dde}\textbf{second-best}} accuracies in each column.}
\label{tab:benchmark-comparison limited}
\end{table}




\bibliographystyle{plain}
\bibliography{bibliography}

\end{document}